\documentclass[12pt,reqno]{article}
\usepackage{amssymb,amscd,amsmath,amsthm}
\newtheorem{theorem}{Theorem}[section]
\newtheorem{proposition}[theorem]{Proposition}
\newtheorem{lemma}[theorem]{Lemma}

\newtheorem{remark}[theorem]{Remark}

\def\cH{\mathcal{H}}
\def\bC{\mathbb{C}}
\def\bR{\mathbb{R}}
\def\bN{\mathbb{N}}
\def\eps{\varepsilon}
\def\<{\langle}
\def\>{\rangle}
\def\ffi{\varphi}

\begin{document}
\baselineskip=16pt

\ \vskip 1cm
\centerline{\bf\LARGE Operator log-convex functions}
\medskip
\centerline{\bf\LARGE and operator means}

\bigskip
\bigskip
\centerline{\Large
Tsuyoshi Ando\footnote{E-mail: ando@es.hokudai.ac.jp}
and Fumio Hiai\footnote{E-mail: hiai@math.is.tohoku.ac.jp}}

\medskip
\begin{center}
$^{1}$\,Shiroishi-ku, Hongo-dori 9, Minami 4-10-805, Sapporo 003-0024, Japan
\end{center}
\begin{center}
$^{2}$\,Graduate School of Information Sciences, Tohoku University, \\
Aoba-ku, Sendai 980-8579, Japan
\end{center}

\medskip
\begin{abstract}
We study operator log-convex functions on $(0,\infty)$, and prove that a continuous
nonnegative function on $(0,\infty)$ is operator log-convex if and only if it is operator
monotone decreasing. Several equivalent conditions related to operator means are given
for such functions. Operator log-concave functions are also discussed.


\end{abstract}

\section*{Introduction}

In 1930's the theory of matrix/operator monotone functions was initiated by
L\"owner \cite{L}, soon followed by the theory of matrix/operator convex functions due to
Kraus \cite{K}. Nearly half a century later, a modern treatment of operator monotone and
convex functions was established by a seminal paper \cite{HP} of Hansen and Pedersen.
Comprehensive expositions on the subject are found in \cite{Do,An,Bh} for example.

Our first motivation to the present paper is the question to determine $\alpha\in\bR$ for
which the functional $\log\omega(A^\alpha)$ is convex in positive operators $A$ for any
positive linear functional $\omega$. In the course of settling the question, we arrived
at the idea to characterize continuous nonnegative functions $f$ on $(0,\infty)$ for
which the operator inequality $f(A\,\triangledown\,B)\le f(A)\,\#\,f(B)$ holds for
positive operators $A$ and $B$, where $A\,\triangledown\,B:=(A+B)/2$ is the arithmetic
mean and $A\,\#\,B$ is the geometric mean \cite{PW,An}. This inequality was indeed
considered by Aujla, Rawla and Vasudeva \cite{ARV} as a matrix/operator version of
log-convex functions. In fact, a function $f$ satisfying the above inequality may be said
to be operator log-convex because the numerical inequality
$f\bigl((a+b)/2)\le\sqrt{f(a)f(b)}$ for $a,b>0$ means the convexity of $\log f$ and the
geometric mean $\#$ is the most standard operator version of geometric mean. Moreover, it
is worth noting that some matrix eigenvalue inequalities involving log-convex functions
were shown in \cite{AB}.

In this paper we show that a continuous nonnegative function $f$ on $(0,\infty)$ is
operator log-convex if and only if it is operator monotone decreasing, and furthermore
present several equivalent conditions related to operator means for the operator
log-convexity. The operator log-concavity counterpart is also considered, and we show that
$f$ is operator log-concave, i.e., $f$ satisfies $f(A\triangledown B)\ge f(A)\,\#\,f(B)$
for positive operators $A,B$ if and only if it is operator monotone (or equivalently,
operator concave).

The paper is organized as follows. In Section 1, after preliminaries on basic notions, the
convexity of $\log\omega(f(A))$ in positive operators $A$ is proved when $f$ is operator
monotone decreasing on $(0,\infty)$. Sections 2 and 3 are the main parts of the paper,
where a number of equivalent conditions are provided for a continuous nonnegative
function on $(0,\infty)$ to be operator log-convex (equivalently, operator monotone
decreasing), or to be operator log-concave (equivalently, operator monotone). In Section 4
another characterization in terms of operator means is provided for a function on
$(0,\infty)$ to be operator monotone.

\section{Operator log-convex functions: motivation}
\setcounter{equation}{0}

In this paper we consider operator monotone and convex functions defined on the half real
line $(0,\infty)$. Let $\cH$ be an infinite-dimensional (separable) Hilbert space. Let
$B(\cH)^+$ denote the set of all positive operators in $B(\cH)$, and $B(\cH)^{++}$ the set
of all invertible $A\in B(\cH)^+$. A continuous real function $f$ on $(0,\infty)$ is said
to be {\it operator monotone} (more precisely, {\it operator monotone increasing}) if
$A\ge B$ implies $f(A)\ge f(B)$ for $A,B\in B(\cH)^{++}$, and {\it operator monotone
decreasing} if $-f$ is operator monotone or $A\ge B$ implies $f(A)\le f(B)$, where $f(A)$
and $f(B)$ are defined via functional calculus as usual. Also, $f$ is said to be
{\it operator convex} if $f(\lambda A+(1-\lambda)B)\le\lambda f(A)+(1-\lambda)f(B)$ for
all $A,B\in B(\cH)^{++}$ and $\lambda\in(0,1)$, and {\it operator concave} if $-f$ is
operator convex. In fact, as easily seen from continuity, the mid-point operator convexity
(when $\lambda=1/2$) is enough for $f$ to be operator convex.

As well known (see \cite[Examples III.2]{An}, \cite[Chapter V]{Bh} for example), a power
function $x^\alpha$ on $(0,\infty)$ is operator monotone (equivalently, operator concave)
if and only if $\alpha\in[0,1]$, operator monotone decreasing if and only if
$\alpha\in[-1,0]$, and operator convex if and only if $\alpha\in[-1,0]\cup[1,2]$.

An axiomatic theory on operator means for operators in $B(\cH)^+$ was developed by Kubo
and Ando \cite{KA} related to operator monotone functions. Corresponding to each
nonnegative operator monotone function $h$ on $[0,\infty)$ with $h(1)=1$ the
{\it operator mean} $\sigma=\sigma_h$ is introduced by
$$
A\,\sigma\,B:=A^{1/2}h(A^{-1/2}BA^{-1/2})A^{1/2},\qquad A,B\in B(\cH)^{++},
$$
which is further extended to $A,B\in B(\cH)^+$ as
\begin{equation}\label{F-1.1}
A\,\sigma\,B:=\lim_{\eps\searrow0}(A+\eps I)\,\sigma\,(B+\eps I)
\end{equation}
in the strong operator topology, where $I$ is the identity operator on $\cH$. The function
$h$ is conversely determined by $\sigma$ as $h(x)=1\,\sigma\,x$ (more precisely,
$h(x)I=I\,\sigma\,xI$) for $x>0$. The following property of operator means is useful:
$$
X^*(A\,\sigma\,B)X=(X^*AX)\,\sigma\,(X^*BX)
$$
for all invertible $X\in B(\cH)$ \cite{KA}.

The most familiar operator means are
\begin{align*}
A\,\triangledown\,B&:={A+B\over2}\quad\mbox{({\it arithmetic mean})}, \\
A\,\#\,B&:=A^{1/2}(A^{-1/2}BA^{-1/2})^{1/2}A^{1/2}\quad\mbox{({\it geometric mean})}, \\
A\,!\,B&:=\biggl({A^{-1}+B^{-1}\over2}\biggr)^{-1}=2(A:B)
\quad\mbox{({\it harmonic mean})}
\end{align*}
for $A,B\in B(\cH)^{++}$ (also for $A,B\in B(\cH)^+$ via \eqref{F-1.1}), where $A:B$ is
the so-called {\it parallel sum}, that is, $A:B:=(A^{-1}+B^{-1})^{-1}$. The geometric mean
was first introduced by Pusz and Woronowicz \cite{PW} in a more general setting for
positive forms. Basic properties of the geometric and the harmonic means for operators are
found in \cite{An}. Note that the operator version of the
{\it arithmetic-geometric-harmonic mean inequality} holds:
$$
A\,\triangledown\,B\ge A\,\#\,B\ge A\,!\,B.
$$

The original motivation to discuss an operator version of log-convex functions came from
the question whether the functional
$$
A\in B(\cH)^{++}\mapsto\log\omega(A^\alpha)
$$
is convex for any $\alpha\in[-1,0]$ and for any positive linear functional $\omega$ on
$B(\cH)$. This is settled by the following:

\begin{proposition}\label{P-1.1}
Let $f$ be a nonnegative operator monotone decreasing function on $(0,\infty)$, and
$\omega$ be a positive linear functional on $B(\cH)$. Then the functional
$$
A\in B(\cH)^{++}\mapsto\log\omega(f(A))\in[-\infty,\infty)
$$
is convex.
\end{proposition}

\begin{proof}
The first part of the proof below is same as the proof of \cite[Proposition 2.1]{ARV}
while we include it for the convenience of the reader. If $f(x)=0$ for some
$x\in(0,\infty)$, then $f$ is identically zero due to analyticity of $f$ (see
\cite[V.4.7]{Bh}) and the conclusion follows trivially. So we assume that $f(x)>0$ for all
$x\in(0,\infty)$. Since $1/f$ is positive and operator monotone on $(0,\infty)$, it
follows (see \cite[Theorem 2.5]{HP}, \cite[V.2.5]{Bh}) that $1/f$ is operator concave on
$(0,\infty)$. Hence
$$
f(A\,\triangledown\,B)^{-1}\ge f(A)^{-1}\,\triangledown\,f(B)^{-1}
$$
so that
\begin{equation}\label{F-1.2}
f(A\,\triangledown\,B)\le f(A)\,!\,f(B),\qquad A,B\in B(\cH)^{++}.
\end{equation}
For each $\lambda>0$, since
$$
f(A)\,!\,f(B)\le f(A)\,\#\,f(B)=(\lambda f(A))\,\#\,(\lambda^{-1}f(B))
\le{\lambda f(A)+\lambda^{-1}f(B)\over2},
$$
we have
$$
\omega(f(A\,\triangledown\,B))\le{\lambda\omega(f(A))+\lambda^{-1}\omega(f(B))\over2},
\qquad A,B\in B(\cH)^{++}.
$$
Minimizing the above right-hand side over $\lambda>0$ yields that
$$
\omega(f(A\,\triangledown B))\le\sqrt{\omega(f(A))\omega(f(B))},
$$
and hence
$$
\log\omega(f(A\,\triangledown B))\le{\log\omega(f(A))+\log\omega(f(B))\over2}.
$$
Since $A\in B(\cH)^{++}\mapsto\log\omega(f(A))\in[-\infty,\infty)$ is continuous in the
operator norm, the convexity follows from the mid-point convexity.
\end{proof}

In the following we state, for convenience, the concave counterpart of Proposition
\ref{P-1.1}. This is immediately seen from the operator concavity of $f$ and the concavity
of $\log x$.

\begin{proposition}\label{P-1.2}
Let $f$ be a nonnegative operator monotone function on $(0,\infty)$, and $\omega$ be a
positive linear functional on $B(\cH)$. Then the functional
$A\in B(\cH)^{++}\mapsto\log\omega(f(A))$ is concave.
\end{proposition}

Let $f$ be a continuous nonnegative function on $(0,\infty)$. An essential point in
the proof of Proposition \ref{P-1.1} is the following operator inequality considered in
\cite{ARV}:
\begin{equation}\label{F-1.3}
f(A\,\triangledown\,B)\le f(A)\,\#\,f(B),\qquad A,B\in B(\cH)^{++}.
\end{equation}
When $f$ satisfies \eqref{F-1.3}, we say that $f$ is {\it operator log-convex}. The term
seems natural because the numerical inequality $f\bigl((a+b)/2)\le\sqrt{f(a)f(b)}$,
$a,b>0$, means the convexity of $\log f$. On the other hand, it is said that $f$ is
{\it operator log-concave} if it satisfies
$$
f(A\,\triangledown\,B)\ge f(A)\,\#\,f(B),\qquad A,B\in B(\cH)^{++}.
$$
Indeed, another operator inequality
\begin{equation}\label{F-1.4}
\log f(A\,\triangledown\,B)\le\{\log f(A)\}\,\triangledown\,\{\log f(B)\},
\qquad A,B\in B(\cH)^{++},
\end{equation}
was also considered in \cite{ARV} for a continuous function $f>0$ on $(0,\infty)$, where
the term ``log matrix convex functions" was referred to \eqref{F-1.4} while
``multiplicatively matrix convex functions" to \eqref{F-1.3}. But we prefer to use
operator log-convexity for \eqref{F-1.3} and we say simply that $\log f$ is operator
convex if $f$ satisfies \eqref{F-1.4} (see Remark \ref{R-3.4} in Section 3 in this
connection).

In the rest of the paper we will prove:
\begin{itemize}
\item[($1^\circ$)] $f$ is operator monotone decreasing if and only if $f$ is operator
log-convex,
\item[($2^\circ$)] $f$ is operator monotone (increasing) if and only if $f$ is operator
log-concave.
\end{itemize}
We will indeed prove results much sharper than ($1^\circ$) and ($2^\circ$), and moreover
present several conditions which are equivalent to those in ($1^\circ$) and ($2^\circ$),
respectively.

\section{Operator monotony, operator log-convexity, and operator means}
\setcounter{equation}{0}

When $f$ is a continuous nonnegative function on $(0,\infty)$, the operator convexity
of $f$ is expressed as
\begin{equation}\label{F-2.1}
f(A\,\triangledown\,B)\le f(A)\,\triangledown\,f(B),\qquad A,B\in B(\cH)^{++}.
\end{equation}
Recall that an operator mean $\sigma$ is said to be {\it symmetric} if
$A\,\sigma\,B=B\,\sigma\,A$ for all $A,B \in B({\cH})^{++}$. Note that the arithmetic mean
$\triangledown$ and the harmonic mean $!$ are the maximum and the minimum symmetric means,
respectively:
\begin{equation}\label{F-2.2}
A\,\triangledown\,B \ge A\,\sigma\,B \ge A\,!\,B,\qquad A,B\in B(\cH)^{++},
\end{equation}
for every symmetric operator mean $\sigma$, or equivalently,
\begin{equation}\label{F-2.3}
{x+1\over2}\ge h(x)\ge{2x\over x+1},\qquad x\ge0,
\end{equation}
for every nonnegative operator monotone function $h$ on $[0,\infty)$ satisfying $h(1)=1$
and the symmetry condition $h(x)=xh(x^{-1})$ for $x>0$ \cite{KA}.

The next theorem characterizes the class of functions $f$ that satisfy the variant of
\eqref{F-2.1} where $\triangledown$ in the right-hand side is replaced with a different
symmetric operator mean. The statement ($1^\circ$) in Section 1 is included in the theorem.

\begin{theorem}\label{T-2.1}
Let $f$ be a continuous nonnegative function on $(0,\infty)$. Then the following
conditions are equivalent:
\begin{itemize}
\item[\rm(a1)] $f$ is operator monotone decreasing;
\item[\rm(a2)] $f(A\,\triangledown\,B)\le f(A)\,\sigma\,f(B)$ for all $A,B\in B(\cH)^{++}$
and for all symmetric operator means $\sigma$;
\item[\rm(a3)] $f$ is operator log-convex, i.e., 
$f(A\,\triangledown\,B)\le f(A)\,\#\,f(B)$ for all $A,B\in B(\cH)^{++}$;
\item[\rm(a4)] $f(A\,\triangledown\,B)\le f(A)\,\sigma\,f(B)$ for all $A,B\in B(\cH)^{++}$
and for some symmetric operator mean $\sigma\ne\triangledown$.
\end{itemize}
\end{theorem}

The following lemma will play a crucial role in proving the theorem.

\begin{lemma}\label{L-2.2}
Let $\ffi$ be a continuous non-decreasing function on $[0,\infty)$ such that
$\ffi(0)=0$ and $\ffi(1)=1$. If a symmetric operator mean $\sigma$ satisfies
$$
\ffi(A\,\triangledown\,B)\le\ffi(A)\,\sigma\,\ffi(B),\qquad A,B\in B(\cH)^{++},
$$
then $\sigma=\triangledown$. {\rm(}Indeed, it is enough to assume that the above
inequality holds for all positive definite $2\times2$ matrices $A,B$.{\rm)}
\end{lemma}

\begin{proof}
Let $P$ and $Q$ be two orthogonal projections in $B(\cH)^+$ such that $P\wedge Q=0$. By
the assumption of the lemma applied to $A_\eps:=P+\eps I$ and $B_\eps:=Q+\eps I$ for
$\eps>0$, we have
$$
\ffi(A_\eps\,\triangledown\,B_\eps)\le\ffi(A_\eps)\,\sigma\,\ffi(B_\eps).
$$
Since $A_\eps\,\triangledown\,B_\eps=P\,\triangledown\,Q+\eps I\to P\,\triangledown\,Q$ in
the operator norm, $\ffi(A_\eps\,\triangledown\,B_\eps)\to\ffi(P\,\triangledown\,Q)$ as
$\eps\searrow0$ in the operator norm. Furthermore, since $\ffi(A_\eps)\searrow\ffi(P)=P$,
$\ffi(B_\eps)\searrow\ffi(Q)=Q$ as $\eps\searrow0$ and the operator mean is continuous in
the strong operator topology under the downward convergence, we have
\begin{equation}\label{F-2.4}
\ffi(P\,\triangledown\,Q)\le P\,\sigma\,Q.
\end{equation}
It follows from \cite[Theorem 3.7]{KA} that $P\,\sigma\,Q=h(0)(P+Q)$, where $h$ is a
symmetric operator monotone function corresponding to $\sigma$. Now choose two orthogonal
projections
$$
P:=\bmatrix1&0\\0&0\endbmatrix,\qquad
Q:=\bmatrix\cos^2\theta&\cos\theta\sin\theta\\\cos\theta\sin\theta&\sin^2\theta\endbmatrix
\quad\mbox{for $0<\theta<\pi/2$}
$$
in the realization of the $2\times2$ matrix algebra in $B(\cH)$. Then $P\wedge Q=0$, and
the diagonalization of $P\,\triangledown\,Q$ is
$$
P\,\triangledown\,Q
=\bmatrix\cos{\theta\over2}&\sin{\theta\over2}\\
\sin{\theta\over2}&-\cos{\theta\over2}\endbmatrix
\bmatrix{1+\cos\theta\over2}&0\\0&{1-\cos\theta\over2}\endbmatrix
\bmatrix\cos{\theta\over2}&\sin{\theta\over2}\\
\sin{\theta\over2}&-\cos{\theta\over2}\endbmatrix.
$$
Therefore,
$$
\ffi(P\,\triangledown\,Q)
=\bmatrix\cos{\theta\over2}&\sin{\theta\over2}\\
\sin{\theta\over2}&-\cos{\theta\over2}\endbmatrix
\bmatrix\ffi\bigl({1+\cos\theta\over2}\bigr)&0\\
0&\ffi\bigl({1-\cos\theta\over2}\bigr)\endbmatrix
\bmatrix\cos{\theta\over2}&\sin{\theta\over2}\\
\sin{\theta\over2}&-\cos{\theta\over2}\endbmatrix.
$$
Comparing the $(1,1)$-entries of both sides of \eqref{F-2.4} we have
$$
\cos^2{\theta\over2}\,\ffi\biggl({1+\cos\theta\over2}\biggr)
+\sin^2{\theta\over2}\,\ffi\biggl({1-\cos\theta\over2}\biggr)
\le h(0)(1+\cos^2\theta)
$$
so that
$$
h(0)\ge{\cos^2{\theta\over2}\,\ffi\bigl({1+\cos\theta\over2}\bigr)
+\sin^2{\theta\over2}\,\ffi\bigl({1-\cos\theta\over2}\bigr)
\over1+\cos^2\theta}.
$$
Letting $\theta\to0$ gives $h(0)\ge1/2$. Since $h(1)=1$ and $h$ is concave, it follows that
$h(x)\ge(x+1)/2$ and so by \eqref{F-2.3} $h(x)=(x+1)/2$ on $[0,1]$, implying
$\sigma=\triangledown$ by analyticity of $h$. The last statement in the parentheses is
obvious from the above proof.
\end{proof}

\noindent
{\it Proof of Theorem \ref{T-2.1}.}\enspace
As shown in the proof of Proposition 1.1, (a1) implies the inequality \eqref{F-1.2}. Hence
(a1) $\Rightarrow$ (a2) holds since the harmonic mean $!$ is the smallest among the
symmetric operator means. It is clear that (a2) $\Rightarrow$ (a3) $\Rightarrow$ (a4). Now
let us prove that (a4) $\Rightarrow$ (a1).

Assume (a4). Since
$$
f(A\,\triangledown\,B)\le f(A)\,\sigma\,f(B)\le f(A)\,\triangledown\,f(B),
\qquad A,B\in B(\cH)^{++},
$$
$f$ is operator convex (hence analytic) on $(0,\infty)$. Hence we may assume that $f(x)>0$
for all sufficiently large $x>0$; otherwise $f$ is identically zero. Since $f(\eps+x)$
obviously satisfies (a4) for any $\eps>0$, we may further assume that the finite limits
$f(+0):=\lim_{x\searrow0}f(x)$ and $f'(+0):=\lim_{x\searrow0}f'(x)$ exist. Then $f$ admits
an integral representation
\begin{equation}\label{F-2.5}
f(x)=\alpha+\beta x+\gamma x^2+\int_{(0,\infty)}
{(\lambda+1)x^2\over \lambda+x}\,d\mu(\lambda),
\end{equation}
where $\alpha,\beta\in\bR$ (indeed, $\alpha=f(+0)$, $\beta=f'(+0)$), $\gamma\ge0$, and
$\mu$ is a finite positive measure on $(0,\infty)$ (see \cite[V.5.5]{Bh}).
In the following we divide the proof into three steps; each step consists of a proof by
contradiction.

{\it Step 1.}\enspace
For $c>0$ large enough so that $f(c)>0$, we write
$$
{f(cx)\over f(c)}={{\alpha\over c^2}+{\beta\over c}x+\gamma x^2
+\int_{(0,\infty)}{(\lambda+1)x^2\over\lambda+cx}\,d\mu(\lambda)\over
{\alpha\over c^2}+{\beta\over c}+\gamma
+\int_{(0,\infty)}{\lambda+1\over\lambda+c}\,d\mu(\lambda)},
$$
and notice that for any fixed $x>0$,
$$
\lim_{c\to\infty}\int_{(0,\infty)}{(\lambda+1)x^2\over\lambda+cx}\,d\mu(\lambda)=0
$$
by the bounded convergence theorem. Suppose, by contradiction, that $\gamma>0$; then we
have
$$
\lim_{c\to\infty}{f(cx)\over f(c)}=x^2,\qquad x>0.
$$
Note that $f_c(x):=f(cx)/f(c)$ satisfies (a4) as well as $f$. Since the operator mean
$\sigma$ is continuous when restricted on the pairs of positive definite matrices, for
every positive definite $2\times2$ matrices $A,B$ (realized in $B(\cH)$) we can take the
limit of $f_c(A\,\triangledown\,B)\le f_c(A)\,\sigma\,f_c(B)$ as $c\to\infty$ to obtain
$(A\,\triangledown\,B)^2\le A^2\,\sigma\,B^2$. By Lemma \ref{L-2.2} for $\ffi(x)=x^2$,
this yields a contradiction with the assumption $\sigma\ne\triangledown$. Hence we must
have $\gamma=0$ so that
$$
f(x)=\alpha+\beta x+\int_{(0,\infty)}{(\lambda+1)x^2\over\lambda+x}\,d\mu(\lambda).
$$

{\it Step 2.}\enspace
For $c>0$ large enough, we write
\begin{equation}\label{F-2.6}
{f(cx)\over f(c)}={{\alpha\over c}+\beta x
+\int_{(0,\infty)}{(\lambda+1)cx^2\over\lambda+cx}\,d\mu(\lambda)\over
{\alpha\over c}+\beta+\int_{(0,\infty)}{(\lambda+1)c\over\lambda+c}\,d\mu(\lambda)}.
\end{equation}
For each fixed $x>0$,  since $(\lambda+1)cx/(\lambda+cx)\nearrow\lambda+1$ as
$c\nearrow\infty$, we notice by the monotone convergence theorem that
$$
\lim_{c\to\infty}\int_{(0,\infty)}{(\lambda+1)cx^2\over\lambda+cx}\,d\mu(\lambda)
=\biggl(\int_{(0,\infty)}(\lambda+1)\,d\mu(\lambda)\biggr)x.
$$
Suppose, by contradiction, that $\int_{(0,\infty)}(\lambda+1)\,d\mu(\lambda)=+\infty$.
For each $c,x\in(0,\infty)$ we set
\begin{equation}\label{F-2.7}
\rho(c,x):={\int_{(0,\infty)}{(\lambda+1)cx\over\lambda+cx}\,d\mu(\lambda)
\over\int_{(0,\infty)}{(\lambda+1)c\over\lambda+c}\,d\mu(\lambda)}.
\end{equation}
Since
\begin{align*}
&{(\lambda+1)c\over\lambda+c}\,x\le{(\lambda+1)cx\over\lambda+cx}
\le{(\lambda+1)c\over\lambda+c}\quad\mbox{if $0<x\le1$}, \\
&{(\lambda+1)c\over\lambda+c}\le{(\lambda+1)cx\over\lambda+cx}
\le{(\lambda+1)c\over\lambda+c}\,x\quad\mbox{if $x\ge1$},
\end{align*}
we notice that for every $c>0$,
\begin{equation}\label{F-2.8}
\begin{cases}
x\le\rho(c,x)\le1 & \text{if $0<x\le1$}, \\
1\le\rho(c,x)\le x & \text{if $x\ge1$},
\end{cases}
\end{equation}
and furthermore $\rho(c,x)$ is non-decreasing in $x>0$ for each fixed $c>0$. Let $D$
denote the countable set of all positive algebraic numbers. Since $\{\rho(c,x):c>0\}$ is
bounded for each fixed $x>0$, one can choose a sequence $\{c_n\}$ with
$0<c_n\nearrow\infty$ such that the limit
\begin{equation}\label{F-2.9}
\kappa(x):=\lim_{n\to\infty}\rho(c_n,x)
\end{equation}
exists for all $x\in D$. Then from \eqref{F-2.6} we obtain
$$
\ffi(x):=x\kappa(x)=\lim_{n\to\infty}{f(c_nx)\over f(c_n)},\qquad x\in D.
$$
Moreover, for each $n$ large enough, since $f_n(x):=f(c_nx)/f(c_n)$ satisfies (a4) and so
$f_n$ is operator convex on $(0,\infty)$, it follows that $\ffi(x)$ is convex on $D$.
Hence $\ffi$ can be extended to a continuous non-decreasing function on $[0,\infty)$,
and it follows from \eqref{F-2.8} that
$$
\begin{cases}
x^2\le\ffi(x)\le x & \text{if $0<x\le1$}, \\
x\le\ffi(x)\le x^2 & \text{if $x\ge1$}.
\end{cases}
$$
In particular, $\ffi(0)=0$ and $\ffi(1)=1$. Now let $A,B$ be positive definite $2\times2$
matrices (realized in $B(\cH)$) whose entries are all rational complex numbers. Since the
eigenvalues of $A$, $B$, and $A\,\triangledown\,B$ are in $D$, we can take the limit of
$f_n(A\,\triangledown\,B)\le f_n(A)\,\sigma\,f_n(B)$ to obtain
\begin{equation}\label{F-2.10}
\ffi(A\,\triangledown\,B)\le\ffi(A)\,\sigma\,\ffi(B).
\end{equation}
Furthermore, we approximate arbitrary positive definite $2\times2$ matrices by those of
rational complex entries and take the limit of \eqref{F-2.10} for approximating matrices
to see that \eqref{F-2.10} holds for all positive definite $2\times2$ matrices $A,B$. Then
Lemma \ref{L-2.2} implies that $\sigma=\triangledown$, a contradiction, so it must follow
that $\int_{(0,\infty)}(\lambda+1)\,d\mu(\lambda)<+\infty$.

{\it Step 3.}\enspace
Finally, suppose, by contradiction, that
$\beta+\int_{(0,\infty)}(\lambda+1)\,d\mu(\lambda)\ne0$. Then it is immediately seen from
\eqref{F-2.6} again that
$$
\lim_{c\to\infty}{f(cx)\over f(c)}=x,\qquad x>0.
$$
By Lemma \ref{L-2.2} for $\ffi(x)=x$, this yields a contradiction again, so we must have
$\beta+\int_{(0,\infty)}(\lambda+1)\,d\mu(\lambda)=0$ so that
$$
f(x)=\alpha+\int_{(0,\infty)}
\biggl\{{(\lambda+1)x^2\over\lambda+x}-(\lambda+1)x\biggr\}\,d\mu(\lambda)
=\alpha-\int_{(0,\infty)}{\lambda(\lambda+1)x\over\lambda+x}\,d\mu(\lambda).
$$
Since
$$
-{x\over\lambda+x}={\lambda\over\lambda+x}-1
$$
is operator monotone decreasing on $(0,\infty)$, so is $f$  and (a1) follows.\qed

\bigskip
The next theorem is the counterpart of Theorem \ref{T-2.1} for operator log-concave
functions, including the statement ($2^\circ$) in Section 1.

\begin{theorem}\label{T-2.3}
Let $f$ be a continuous nonnegative function on $(0,\infty)$. Then the following
conditions are equivalent:
\begin{itemize}
\item[\rm(b1)] $f$ is operator monotone;
\item[\rm(b2)] $f(A\,\triangledown\,B)\ge f(A)\,\sigma\,f(B)$ for all $A,B\in B(\cH)^{++}$
and for all symmetric means $\sigma$;
\item[\rm(b3)] $f$ is operator log-concave, i.e.,
$f(A\,\triangledown\,B)\ge f(A)\,\#\,f(B)$ for all $A,B\in B(\cH)^{++}$;
\item[\rm(b4)] $f(A\,\triangledown\,B)\ge f(A)\,\sigma\,f(B)$ for all $A,B\in B(\cH)^{++}$
and for some symmetric operator mean $\sigma\ne\,!$.
\end{itemize}
\end{theorem}

We need the following lemma to prove the theorem.

\begin{lemma}\label{L-2.4}
Let $f$ be a continuous nonnegative function on $(0,\infty)$, and assume that
\begin{equation}\label{F-2.11}
f(A\,\triangledown\,B)\ge f(A)\,!\,f(B),\qquad A,B\in B(\cH)^{++}.
\end{equation}
Then, either $f(x)>0$ for all $x>0$ or $f$ is identically zero. {\rm(}Indeed, it is enough
to assume that the above inequality holds for all positive definite $2\times2$ matrices
$A,B$.{\rm)}
\end{lemma}

\begin{proof}
Assume that $f(x)=0$ for some $x>0$ but $f$ is not identically zero. The assumption
\eqref{F-2.11} applied to $A=aI$ and $B=bI$ gives $f(a\,\triangledown\,b)\ge f(a)\,!\,f(b)$
for every scalars $a,b>0$. By induction on $n\in\bN$ one can easily see that
\begin{equation}\label{F-2.12}
f((1-\lambda)a+\lambda b)\ge f(a)\,!_\lambda\,f(b)
\end{equation}
for all  $a,b>0$ and all $\lambda=k/2^n$, $k=0,1,\dots,2^n$, $n\in\bN$, where
$u\,!_\lambda\,v$ with $0\le\lambda\le1$ is the $\lambda$-harmonic mean for scalars
$u,v\ge0$ defined as
$$
u\,!_\lambda\,v:=\lim_{\eps\searrow0}
\bigl((1-\lambda)(u+\eps)^{-1}+\lambda(v+\eps)^{-1}\bigr)^{-1}.
$$
Furthermore, thanks to the continuity of $f$, \eqref{F-2.12} holds for all $a,b>0$
and all $\lambda\in[0,1]$. So we notice that $f(x)>0$ for all $x$ between $a,b$ whenever
$f(a)>0$ and $f(b)>0$. Thus it follows from the assumption on $f$ that there is an
$\alpha\in(0,\infty)$ such that the following (i) or (ii) holds:
\begin{itemize}
\item[(i)] $f(x)=0$ for all $x\in(0,\alpha]$ and $f(x)>0$ for all
$x\in(\alpha,\alpha+\delta]$ for some $\delta>0$,
\item[(ii)] $f(x)>0$ for all $x\in(0,\alpha)$ and $f(x)=0$ for all $x\in[\alpha,\infty)$.
\end{itemize}

Let $H$ and $K$ be $2\times2$ Hermitian matrices in the realization of the $2\times2$
matrix algebra in $B(\cH)$. For every $\gamma\in\bR$ such that $\alpha I+\gamma H$,
$\alpha I+\gamma K>0$ (i.e., positive definite), one can apply \eqref{F-2.11} to
$A:=\alpha I+\gamma H$ and $B:=\alpha I+\gamma K$ to obtain
\begin{equation}\label{F-2.13}
f\biggl(\alpha I+\gamma\,{H+K\over2}\biggr)
\ge f(\alpha I+\gamma H)\,!\,f(\alpha I+\gamma K).
\end{equation}
Write for short
$$
X:=f(\alpha I+\gamma H),\quad Y:=f(\alpha I+\gamma K),\quad
Z:=f\biggl(\alpha I+\gamma\,{H+K\over2}\biggr),
$$
and let $s(X)$, $s(Y)$, and $s(Z)$ denote the support projections of $X$, $Y$, and
$Z$, respectively, that is, the orthogonal projections onto the ranges of $X$, $Y$, and
$Z$ (in $\bC^2$), respectively. Since $X\ge\eps s(X)$ and $Y\ge\eps s(Y)$ for a
sufficiently small $\eps>0$, \eqref{F-2.13} implies that
$$
Z\ge\{\eps s(X)\}\,!\,\{\eps s(Y)\}=\eps\{s(X)\wedge s(Y)\}.
$$
Letting $P:=s(X)\wedge s(Y)$ we have
$$
0=(I-s(Z))Z(I-s(Z))\ge\eps(I-s(Z))P(I-s(Z))
$$
so that $P(I-s(Z))=0$ or equivalently $P\le s(Z)$. Therefore,
$$
s(Z)\ge s(X)\wedge s(Y).
$$
For each Hermitian matrix $S$ let $S=S_+-S_-$ be the Jordan decomposition of $S$. In the
case (i) choose a $\gamma>0$ small enough so that $\alpha I+\gamma H$,
$\alpha I+\gamma K\le(\alpha+\delta)I$, and in the case (ii) choose a $\gamma<0$ so that
$\alpha I+\gamma H$, $\alpha I+\gamma K>0$. Then we have
$$
s(X)=s(H_+),\quad s(Y)=s(K_+),\quad s(Z)=s((H+K)_+)
$$
and so
\begin{equation}\label{F-2.14}
s((H+K)_+)\ge s(H_+)\wedge s(K_+).
\end{equation}

Thus, to prove the lemma by contradiction, it suffices to show that \eqref{F-2.14} is not
true in general. We notice that \eqref{F-2.14} yields
\begin{equation}\label{F-2.15}
s(H_+)\ge s(K_+)\quad\mbox{whenever $H>K$}.
\end{equation}
In fact, letting $G:=H-K>0$ (hence $s(G_+)=s(G)=I$) we have
$$
s(H_+)=s((G+K)_+)\ge s(G_+)\wedge s(K_+)=s(K_+).
$$
Hence it suffices to show that \eqref{F-2.15} is not true in general. Now let
$P:=\bmatrix1&0\\0&0\endbmatrix$ and $Q:=\bmatrix1/2&1/2\\1/2&1/2\endbmatrix$, and define
$H:=P$ and $K:=\eps Q-(I-Q)$ for $\eps>0$. Then $s(H_+)=P\not\ge Q=s(K_+)$. But since
$$
H-K=\bmatrix1&0\\0&0\endbmatrix-\eps\bmatrix1/2&1/2\\1/2&1/2\endbmatrix
+\bmatrix1/2&-1/2\\-1/2&1/2\endbmatrix
=\bmatrix{3-\eps\over2}&-{1+\eps\over2}\\-{1+\eps\over2}&{1-\eps\over2}\endbmatrix
$$
and
$$
\det(H-K)=\biggl({3-\eps\over2}\biggr)\biggl({1-\eps\over2}\biggr)
-\biggl({1+\eps\over2}\biggr)^2
={1-3\eps\over2},
$$
we have $H>K$ for small $\eps>0$. Hence \eqref{F-2.15} is not true. The last statement in
the parentheses is obvious from the above proof.
\end{proof}

\noindent
{\it Proof of Theorem \ref{T-2.3}.}\enspace
Assume (b1); then $f$ is operator concave \cite[Theorem 2.5]{HP}, so (b2) follows. It is
obvious that (b2) $\Rightarrow$ (b3) $\Rightarrow$ (b4). Finally, let us prove that
(b4) $\Rightarrow$ (b1). Since (b4) implies the assumption of Lemma \ref{L-2.4}, we may
assume by Lemma \ref{L-2.4} that $f(x)>0$ for all $x>0$. Then (b4) implies that
$$
f(A\,\triangledown\,B)^{-1}\le(f(A)\,\sigma\,f(B))^{-1}
=f(A)^{-1}\,\sigma^*\,f(B)^{-1},\qquad A,B\in B(\cH)^{++},
$$
where $\sigma^*$ is the adjoint of $\sigma$, the symmetric operator mean defined by
$A\,\sigma^*\,B:=(A^{-1}\,\sigma\,B^{-1})^{-1}$ \cite{KA}. Since $\sigma\ne\,!$
means that $\sigma^*\ne\triangledown$, Theorem \ref{T-2.1} implies that $1/f$ is operator
monotone decreasing, so (b1) follows.\qed

\begin{remark}\label{R-2.5}\rm
By Lemma \ref{L-2.4} it is also seen that a continuous nonnegative function $f$ on
$(0,\infty)$ satisfies \eqref{F-2.11} if and only if $f$ is identically zero, or $f>0$ and
$1/f$ is operator convex.
\end{remark}

\begin{remark}\label{R-2.6}\rm
For each $\lambda\in[0,1]$ the $\lambda$-arithmetic and the $\lambda$-harmonic means are
$A\,\triangledown_\lambda\,B:=(1-\lambda)A+\lambda B$ and
$A\,!_\lambda\,B:=((1-\lambda)A^{-1}+\lambda B^{-1})^{-1}$ for $A,B\in B(\cH)^{++}$.
Let $\sigma$ be an operator mean corresponding to an operator monotone function $h$ on
$[0,\infty)$ such that $h'(1)=\lambda$. Then we have
$A\,\triangledown_\lambda\,B\ge A\,\sigma\,B\ge A\,!_\lambda\,B$ extending \eqref{F-2.2}.
As in the proof of Proposition \ref{P-1.1},
$$
f(A\,\triangledown_\lambda\,B)\le f(A)\,!_\lambda\,f(B)
\le f(A)\,\sigma\,f(B),\qquad A,B\in B(\cH)^{++},
$$
whenever $f\ge0$ is operator monotone decreasing on $(0,\infty)$. Consequently, for such a
function $f$,
\begin{equation}\label{F-2.16}
f(A\,\triangledown_\lambda\,B)\le f(A)\,\#_\lambda\,f(B),
\qquad A,B\in B(\cH)^{++},
\end{equation}
where $\#_\lambda$ is the $\lambda$-power mean corresponding to the power function
$x^\lambda$. The reversed inequality of \eqref{F-2.16} holds if $f$ is operator monotone.
We may adopt \eqref{F-2.16} for the definition of operator log-convexity. Indeed, if $f$
is a nonnegative function (not assumed to be continuous) on $(0,\infty)$ and satisfies
\eqref{F-2.16} for all positive definite $n\times n$ matrices $A,B$ of every $n$, then $f$
is continuous and a standard convergence argument shows that $f$ is operator log-convex. 
\end{remark}

\begin{remark}\label{R-2.7}\rm
The arithmetic and the harmonic means of $n$ operators $A_1,\dots,A_n$ in $B(\cH)^{++}$
are
$$
\mathbf{A}(A_1,\dots,A_n):={A_1+\dots+A_n\over n},\quad
\mathbf{H}(A_1,\dots,A_n):=\biggl({A^{-1}+\dots+A_n^{-1}\over n}\biggr)^{-1}.
$$
The geometric mean $\mathbf{G}(A_1,\dots,A_n)$ for $n\ge3$ was rather recently introduced
in \cite{ALM} in a recursive way. (A different notion of geometric means for $n$ operators
is in \cite{BH}.) From the arithmetic-geometric-harmonic mean inequality for $n$ operators
in \cite{ALM}, we have
$$
f(\mathbf{A}(A_1,\dots,A_n))\le\mathbf{H}(f(A_1),\dots,f(A_n))
\le\mathbf{G}(f(A_1),\dots,f(A_n))
$$
if $f\ge0$ is operator monotone decreasing on $(0,\infty)$, and if $f$ is operator
monotone,
$$
f(\mathbf{A}(A_1,\dots,A_n))\ge\mathbf{A}(f(A_1),\dots,f(A_n))
\ge\mathbf{G}(f(A_1),\dots,f(A_n)).
$$
\end{remark}

\section{Further characterizations}
\setcounter{equation}{0}

In this section we present further conditions equivalent to those of Theorems \ref{T-2.1}
and \ref{T-2.3}, respectively. To exclude the singular case of identically zero function
and thus make statements simpler, we assume throughout the section that $f$ is a
continuous positive (i.e., $f(x)>0$ for all $x>0$) function on $(0,\infty)$.

\begin{theorem}\label{T-3.1}
For a continuous positive function $f$ on $(0,\infty)$, each of the following
conditions {\rm(a5)}--{\rm(a13)} is equivalent to {\rm(a1)}--{\rm(a4)} of Theorem
\ref{T-2.1}:
\begin{itemize}
\item[\rm(a5)] $\bmatrix f(A)&f(A\,\triangledown\,B)\\f(A\,\triangledown\,B)&f(B)
\endbmatrix\ge0$ for all $A,B\in B(\cH)^{++}$, where
$\bmatrix X_{11}&X_{12}\\X_{21}&X_{22}\endbmatrix$ for $X_{ij}\in B(\cH)$ is considered
as an operator in $B(\cH\oplus\cH)$ as usual;
\item[\rm(a6)] $f(A\,\triangledown\,B)f(B)^{-1}f(A\,\triangledown\,B)\le f(A)$ for all
$A,B\in B(\cH)^{++}$;
\item[\rm(a7)] $f(A\,\triangledown\,B)\le{1\over2}\{\lambda f(A)+\lambda^{-1}f(B)\}$ for
all $A,B\in B(\cH)^{++}$ and all $\lambda>0$;
\item[\rm(a8)] $A\in B(\cH)^{++}\mapsto\log\<\xi,f(A)\xi\>$ is convex for every
$\xi\in\cH$;
\item[\rm(a9)] $(A,\xi)\mapsto\<\xi,f(A)\xi\>$ is jointly convex for $A\in B(\cH)^{++}$
and $\xi\in\cH$;
\item[\rm(a10)] $f$ is operator convex and the numerical function $\log f(x)$ is convex;
\item[\rm(a11)] both $f$ and $\log f$ are operator convex;
\item[\rm(a12)] $f$ is operator convex and the numerical function $f(x)$ is non-increasing;
\item[\rm(a13)] $f$ admits a representation
\begin{equation}\label{F-3.1}
f(x)=\alpha+\int_{[0,\infty)}{\lambda+1\over\lambda+x}\,d\mu(\lambda),
\end{equation}
where $\alpha\ge0$ and $\mu$ is a finite positive measure on $[0,\infty)$.
\end{itemize}
\end{theorem}

Before proving the theorem we give the next lemma, which may be of independent interest.

\begin{lemma}\label{L-3.2}
Let $\ffi(x)$ be a continuous non-decreasing function on $(0,\infty)$ such that
$\ffi(0)=0$ and $\ffi(1)=1$. Then $(A,\xi)\mapsto\<\xi,\ffi(A)\xi\>$ for
$A\in B(\cH)^{++}$ and $\xi\in\cH$ cannot be jointly convex. {\rm(}Indeed, this functional
cannot be jointly convex even when $A$ is restricted to positive definite $2\times2$
matrices and $\xi$ to vectors in $\bC^2$.{\rm)}
\end{lemma}

\begin{proof}
First, recall the well-known expression for the parallel sum:
\begin{equation}\label{F-3.2}
\<\xi,(A:B)\xi\>=\inf\{\<\xi_1,A\xi_1\>+\<\xi_2,B\xi_2\>:
\xi=\xi_1+\xi_2,\ \xi_1,\xi_2\in\cH\}
\end{equation}
for any $A,B\in B(\cH)^{++}$ and $\xi\in\cH$ (see \cite[Theorem I.3]{An} for example).
Suppose, by contradiction, that the functional in question is jointly convex. Let us show
that
\begin{equation}\label{F-3.3}
\ffi(A\,\triangledown\,B)\le\ffi(A)\,!\,\ffi(B),\qquad A,B\in B(\cH)^{++}.
\end{equation}
For any decomposition $\xi=\xi_1+\xi_2$ of $\xi\in\cH$ we have 
\begin{align*}
\<\xi,\ffi(A\,\triangledown\,B)\xi\>
&=4\biggl\<{\xi_1+\xi_2\over2},\ffi\biggl({A+B\over2}\biggr)
\biggl({\xi_1+\xi_2\over2}\biggr)\biggr\> \\
&\le2\{\<\xi_1,\ffi(A)\xi_1\>+\<\xi_2,\ffi(B)\xi_2\>\},
\end{align*}
which implies by \eqref{F-3.2} that
$$
\<\xi,\ffi(A\,\triangledown\,B)\xi\>\le\<\xi,(\ffi(A)\,!\,\ffi(B))\xi\>.
$$
Hence \eqref{F-3.3} follows, which yields a contradiction by Lemma \ref{L-2.2}.
\end{proof}

\noindent
{\it Proof of Theorem \ref{T-3.1}.}\enspace
(a5) $\Leftrightarrow$ (a6) is well known (see \cite[Theorem I.1]{An}, \cite[1.3.3]{Bh2}).
(a5) $\Rightarrow$ (a3) follows from the following characterization of the geometric mean
given in \cite[Theorem I.2 and its proof]{An}:
$$
X\,\#\,Y=\max\biggl\{Z\in B(\cH)^+:\bmatrix X&Z\\Z&Y\endbmatrix\ge0\biggr\}
\quad\mbox{for $X,Y\in B(\cH)^+$}.
$$
The implications (a3) $\Rightarrow$ (a7) $\Rightarrow$ (a8) were already shown in the proof
of Proposition \ref{P-1.1}.

(a8) $\Rightarrow$ (a10).\enspace
The operator convexity of $f$ is immediate because $f$ is operator convex if (and only if)
$A\in B(\cH)^{++}\mapsto\<\xi,f(A)\xi\>$ is convex for every $\xi\in\cH$. The convexity of
$\log f(x)$ is also obvious by taking $A=aI$ in (a8).

(a10) $\Rightarrow$ (a1).\enspace
This can be shown in a manner similar to the three-stepped proof of
(a4) $\Rightarrow$ (a1) of Theorem \ref{T-2.1}. By considering $f(\eps+x)$ for each
$\eps>0$, we may assume that $f$ admits the representation \eqref{F-2.5}. For Step 1,
suppose that $\gamma>0$; then we have $\lim_{c\to\infty}f(cx)/f(c)=x^2$ for all $x>0$.
Since $\log f(cx)$ is convex by assumption, the limit function $2\log x$ is convex as well,
which is absurd. Hence $\gamma=0$.

For Step 2, suppose that $\int_{(0,\infty)}(\lambda+1)\,d\mu(\lambda)=+\infty$. One can
choose a sequence $\{c_n\}$ with $0<c_n\nearrow\infty$ such that the limit $\kappa(x)$ in
\eqref{F-2.9}, with $\rho(c,x)$ in \eqref{F-2.7}, exists for all rational numbers $x>0$.
From \eqref{F-2.8} and \eqref{F-2.6} we have $1\le\kappa(x)\le x$ for all rational $x\ge1$
and $\ffi(x):=x\kappa(x)=\lim_{n\to\infty}f(c_nx)/f(c_n)$ for all rational $x>0$. Since
$\log f(c_nx)$ is convex on $(0,\infty)$, it follows that $\log\ffi(x)$ is convex on the
rational numbers $x\ge1$. Hence $\ffi$ can be extended to a continuous function on
$[1,\infty)$ so that $\psi(x):=\log\ffi(x)$ is convex on $[1,\infty)$ and
\begin{equation}\label{F-3.4}
\log x\le\psi(x)\le2\log x,\qquad x\ge1.
\end{equation}
For any $a\ge1$, by convexity of $\psi$ we have
$$
{\psi(a)\over a}\le\lim_{x\to\infty}{\psi(x)\over x}
\le2\lim_{x\to\infty}{\log x\over x}=0.
$$
Hence $\psi(a)=0$ for all $a\ge1$, which contradicts the first inequality in \eqref{F-3.4}.
Hence $\int_{(0,\infty)}(\lambda+1)\,d\mu(\lambda)<+\infty$.

Step 3 here is the same as that in the proof of (a4) $\Rightarrow$ (a1) of Theorem
\ref{T-2.1} by considering the limit function $\log x$ of $\log(f(cx)/f(c))$ as
$c\to\infty$.

(a1) $\Rightarrow$ (a13).\enspace
This implication was shown in the proof of the main theorem of \cite{Ha1}, and
the converse is obvious. We state (a13) since it is useful to derive (a5) from (a1). The
following proof is slightly simpler than that in \cite{Ha1}.
Since (a1) is equivalent to $f(x^{-1})$ being operator monotone, we have a representation
\begin{equation}\label{F-3.5}
f(x^{-1})=\alpha+\beta x
+\int_{(0,\infty)}{(\lambda+1)x\over\lambda+x}\,d\nu(\lambda),
\end{equation}
where $\alpha,\beta\ge0$ and $\nu$ is a positive finite measure on $(0,\infty)$
\cite[pp.\ 144-145]{Bh}. By taking $d\mu(\lambda):=d\nu(\lambda^{-1})$ on $(0,\infty)$
and by extending it to a measure on $[0,\infty)$ with $\mu(\{0\})=\beta$, the
representation \eqref{F-3.5} is transformed into \eqref{F-3.1}.

(a13) $\Rightarrow$ (a5).\enspace
Thanks to (a5) $\Leftrightarrow$ (a6) as mentioned above, it suffices to show that the
component functions $f_1(x):=\alpha$, $f_2(x):=1/x$, and $f_3(x):=1/(x+\lambda)$ for
$\lambda>0$ in the expression \eqref{F-3.1} satisfy the inequality in (a6). It is trivial
for $f_1$. For $f_2$ we have to show that
$$
\biggl({A+B\over2}\biggr)^{-1}B\biggl({A+B\over2}\biggr)^{-1}\le A^{-1},
$$
or equivalently,
\begin{equation}\label{F-3.6}
\biggl({A+B\over2}\biggr)B^{-1}\biggl({A+B\over2}\biggr)\ge A.
\end{equation}
With $C:=B^{-1/2}AB^{-1/2}$, \eqref{F-3.6} is further reduced to ${1\over4}(C+I)^2\ge C$,
which obviously holds. The assertion for $f_3$ follows from that for $f_2$ by taking
$A+\lambda I$ and $B+\lambda I$ in place of $A$ and $B$.

Now, conditions (a9), (a11), and (a12) are outside the above proved circle of
equivalence, whose equivalence to (a1) is proved below.

(a1) $\Leftrightarrow$ (a11).\enspace
Since (a1) implies that $1/f$ is operator monotone and since $\log x$ is operator monotone
on $(0,\infty)$, it is immediate to see that $\log(1/f)=-\log f$ is operator monotone. This
implies that $-\log f$ is operator concave or $\log f$ is operator convex. For the
converse, (a11) $\Rightarrow$ (a10) is trivial.

(a1) $\Leftrightarrow$ (a9).\enspace
The implication (a13) $\Rightarrow$ (a9) was shown in \cite[Remark 4.6]{Ha2}. The proof
of (a9) $\Rightarrow$ (a1) can be done similarly to (a4) $\Rightarrow$ (a1) of Theorem
\ref{T-2.1} by dividing into three steps. First, from the fact mentioned in the proof of
(a8) $\Rightarrow$ (a10), we may assume that $f$ admits the representation \eqref{F-2.5}.
Then for Steps 1 and 3, we may only notice that the functions $x^2$ and $x$ do not satisfy
(a9) as particular cases of Lemma \ref{L-3.2}. For Step 2, suppose that
$\int_{(0,\infty)}(\lambda+1)\,d\mu(\lambda)=+\infty$; then as in the proof of
(a4) $\Rightarrow$ (a1) there is a sequence $c_n\nearrow\infty$ such that
$\ffi(x):=\lim_{n\to\infty}f(c_nx)/f(c_n)$ exists for all algebraic numbers $x>0$, and
$\ffi$ can be extended to a continuous non-decreasing function on $[0,\infty)$ with
$\ffi(0)=0$ and $\ffi(1)=1$. Furthermore, since $f(c_nx)$ satisfies (a9), it follows as
in the proof of (a4) $\Rightarrow$ (a1) that $\ffi$ satisfies (a9) as well when $A$ is
restricted to positive definite $2\times2$ matrices. This yields a contradiction by Lemma
\ref{L-3.2}, which shows that $\int_{(0,\infty)}(\lambda+1)\,d\mu(\lambda)<+\infty$.

(a1) $\Leftrightarrow$ (a12).\enspace
The implication (a1) $\Rightarrow$ (a12) is immediate since (a1) implies the operator
convexity of $f$. The converse can be proved once again similarly to (a10) $\Rightarrow$
(a1); just use the non-increasingness of $f(cx)$ instead of the convexity of $\log f(cx)$.
In fact, for Step 2, if we suppose that
$\int_{(0,\infty)}(\lambda+1)\,d\mu(\lambda)=+\infty$, then the function $\varphi(x)$
defined and extended as above is non-increasing by the assumption (a12) as well as
non-decreasing with $\varphi(x)\ge x$ for $x\ge1$ (by the definition of $\ffi$). This is
a contradiction.\qed

\begin{remark}\label{R-3.3}\rm
Let $\Phi:B(\cH)\to B(\mathcal{K})$ be a positive linear map, where $\mathcal{K}$ is
another Hilbert space. If $f$ is operator log-convex on $(0,\infty)$, then we have
$$
\Phi(f(A\,\triangledown\,B))\le\Phi(f(A)\,\#\,f(B))
\le\Phi(f(A))\,\#\,\Phi(f(B))
$$
for all $A,B\in B(\cH)^+$ thanks to \cite[Corollary IV.1.3]{An}. This in particular gives
another proof of (a3) $\Rightarrow$ (a8) by taking a positive linear functional as $\Phi$.
\end{remark}

\begin{remark}\label{R-3.4}\rm
The implication (a3) $\Rightarrow$ (a11) says that \eqref{F-1.3} implies \eqref{F-1.4},
that is, the operator log-convexity of $f$ implies that $\log f$ is operator convex. This
may also justify our term operator log-convexity.
\end{remark}

\begin{remark}\label{R-3.5}\rm
In \cite[Remark 4.6]{Ha2} Hansen posed the question to characterize functions $f$ on
$(0,\infty)$ for which condition (a9) holds. By taking $A=aI$ in $\<\xi,f(A)\xi\>$ for any
fixed $a\in(0,\infty)$, it is clear that $f$ must be nonnegative whenever it satisfies
(a9). Consequently, Theorem \ref{T-3.1} settles the above question as follows: A continuous
function $f$ on $(0,\infty)$ satisfies (a9) if and only if $f$ is nonnegative and operator
monotone decreasing, or equivalently, $f$ admits a representation in (a13).
\end{remark}

\begin{remark}\label{R-3.6}\rm
In \cite{Uc} Uchiyama recently proved that a continuous (not necessarily positive) function
$f$ on $(0,\infty)$ is operator monotone decreasing if and only if it is operator convex
and $f(\infty):=\lim_{x\to\infty}f(x)<+\infty$. This implies that (a1) $\Leftrightarrow$
(a13), because the non-increasingness of a convex function $f$ on
$(0,\infty)$ is equivalent to $f(\infty)<+\infty$.
\end{remark}

The following is the concave counterpart of Theorem \ref{T-3.1}, which is easily shown by
converting corresponding conditions of Theorem \ref{T-3.1}.

\begin{theorem}\label{T-3.6}
For a continuous positive function $f$ on $(0,\infty)$, each of the following
conditions {\rm(b5)}--{\rm(b10)} is equivalent to {\rm(b1)}--{\rm(b4)} of Theorem
\ref{T-2.3}:
\begin{itemize}
\item[\rm(b5)] $\bmatrix f(A)&f(A\,!\,B)\\f(A\,!\,B)&f(B)
\endbmatrix\ge0$ for all $A,B\in B(\cH)^{++}$;
\item[\rm(b6)] $f(A\,\triangledown\,B)f(B)^{-1}f(A\,\triangledown\,B)\ge f(A)$ for all
$A,B\in B(\cH)^{++}$;
\item[\rm(b7)] $f(A\,!\,B)\le{1\over2}\{\lambda f(A)+\lambda^{-1}f(B)\}$ for
all $A,B\in B(\cH)^{++}$ and all $\lambda>0$;
\item[\rm(b8)] $A\in B(\cH)^{++}\mapsto\log\<\xi,f(A)\xi\>$ is concave for
every $\xi\in\cH$;
\item[\rm(b9)] $f$ is operator concave;
\item[\rm(b10)] $f$ admits a representation
$$
f(x)=\alpha+\beta x+\int_{(0,\infty)}{(\lambda+1)x\over\lambda+x}\,d\mu(\lambda),
$$
where $\alpha,\beta\ge0$ and $\mu$ is a finite positive measure on $(0,\infty)$.
\end{itemize}
\end{theorem}

\begin{proof}
Since $f$ satisfies (b1) if and only if $1/f$ (or $f(x^{-1})$) satisfies (a1), each
condition of Theorem \ref{T-3.1} for $1/f$ (or $f(x^{-1})$) instead of $f$ is equivalent
to (b1). (b5) and (b7) are (a5) and (a7) for $f(x^{-1})$, respectively. Also, (b6) is
(a6) for $1/f$.

The implication (b1) $\Rightarrow$ (b8) is a particular case of Proposition \ref{P-1.2}.
Conversely, assume (b8). For every $A\in B(\cH)^{++}$ and $\xi\in\cH$ notice that
$$
\<\xi,f(A)^{-1}\xi\>=\sup_{\eta\ne0}{|\<\xi,\eta\>|^2\over\<\eta,f(A)\eta\>}
$$
and so
$$
\log\<\xi,f(A)^{-1}\xi\>
=\sup_{\eta\ne0}\bigl\{2\log|\<\xi,\eta\>|-\log\<\eta,f(A)\eta\>\bigr\}.
$$
Since (b8) implies that $A\in B(\cH)^{++}\mapsto2\log|\<\xi,\eta\>|-\log\<\eta,f(A)\eta\>$
is convex, it follows that $1/f$ satisfies (a8). Hence (b8) $\Rightarrow$ (b1).

Finally, (b1) $\Leftrightarrow$ (b9) and (b1) $\Leftrightarrow$ (b10) are well known
\cite{Bh,HP}, which were indeed used in the proofs of Theorems \ref{T-2.3} and \ref{T-3.1}.
We state (b9) and (b10) just for the sake of completeness.
\end{proof}

\section{More about operator monotony and operator \\ means}
\setcounter{equation}{0}

When $f$ is an operator monotone (not necessarily nonnegative) function on $(0,\infty)$,
it is obvious that
$$
f(A\,\triangledown\,B)\ge f(A\,\#\,B)\ge f(A\,!\,B),\qquad A,B\in B(\cH)^{++}.
$$
In the next proposition we show that an inequality such as
$f(A\,\triangledown\,B)\ge f(A\,\#\,B)$ for all $A,B\in B(\cH)^{++}$
conversely implies the operator monotony of $f$, thus giving yet another characterization
of operator monotone functions on $(0,\infty)$ in terms of operator means.

\begin{proposition}\label{P-4.1}
A continuous function $f$ on $(0,\infty)$ is operator monotone if and only if one of the
following conditions holds:
\begin{itemize}
\item[\rm(1)] $f(A\,\triangledown\,B)\ge f(A\,\sigma\,B)$ for all $A,B\in B(\cH)^{++}$ and
for some symmetric operator mean $\sigma\ne\triangledown$;
\item[\rm(2)] $f(A\,!\,B)\le f(A\,\sigma\,B)$ for all $A,B\in B(\cH)^{++}$ and for
some symmetric operator mean $\sigma\ne\,!$.
\end{itemize}
The operator monotone decreasingness of $f$ is equivalent to each of {\rm(1)} and {\rm(2)}
with the reversed inequality.
\end{proposition}

Note by \eqref{F-2.2} that the inequalities in (1) and (2) actually hold for all symmetric
operator means if $f$ is operator monotone. We first prove the next lemma.

\begin{lemma}\label{L-4.2}
Let $\sigma$ be a symmetric operator mean such that $\sigma\ne\triangledown$, and let
$\gamma_0:=2\,\sigma\,0$. If $X,Y\in B(\cH)^{++}$ and $X\ge Y\ge\gamma X$ with
$\gamma\in(\gamma_0,1]$, then there exist $A,B\in B(\cH)^{++}$ such that
$X=A\,\triangledown\,B$ and $Y=A\,\sigma\,B$.
\end{lemma}

\begin{proof}
Let $h$ be the operator monotone function on $[0,\infty)$ corresponding to $\sigma$, i.e.,
$h(x):=1\,\sigma\,x$ for $x\ge0$. Since $\gamma_0=2h(0)$, we have $0\le\gamma_0\le1$ by
\eqref{F-2.3}. Note that $h(0)=1/2$ implies $\sigma=\triangledown$ (see the last part of
the proof of Lemma \ref{L-2.2}). Hence we have $0\le\gamma_0<1$.

Note that $X\ge Y\ge\gamma X$ is equivalent to $I\ge X^{-1/2}YX^{-1/2}\ge\gamma I$. When
we have $A,B\in B(\cH)^{++}$ such that $I=A\,\triangledown\,B$ and
$X^{-1/2}YX^{-1/2}=A\,\sigma\,B$, it follows that
$X=(X^{1/2}AX^{1/2})\,\triangledown\,(X^{1/2}BX^{1/2})$ and
$Y=(X^{1/2}AX^{1/2})\,\sigma\,(X^{1/2}BX^{1/2})$. Thus we may assume that
$I\ge Y\ge\gamma I$ with $\gamma\in(\gamma_0,1]$ and find $A,B\in B(\cH)^{++}$ such that
$I=A\,\triangledown\,B$ and $Y=A\,\sigma\,B$. For this, it suffices to find an
$A\in B(\cH)^{++}$ such that $A\le I$ and $A\,\sigma\,(2I-A)=Y$. Define
$\ffi(t):=t\,\sigma\,(2-t)$ for $0\le t\le1$; then for $0<t\le1$ we have
$\ffi(t)=th(2t^{-1}-1)$ and so
$$
\ffi'(t)=h(2t^{-1}-1)-2t^{-1}h'(2t^{-1}-1).
$$
Letting $a:=2t^{-1}-1\in(1,\infty)$ for any $t\in(0,1)$, one can see that
$h'(a)<(h(a)-1)/(a-1)$. In fact, suppose on the contrary that $h'(a)\ge(h(a)-1)/(a-1)$;
then by concavity $h$ must be linear on $[1,a]$. Furthermore, $h'(1)=1/2$ since $\sigma$
is symmetric, that is, $h(x)=xh(x^{-1})$ for $x>0$. Hence it follows that $h(x)=(x+1)/2$
on $[1,a]$, implying $\sigma=\triangledown$. Therefore we have
$$
h'(a)<{h(a)-1\over a-1}\le{h(a)\over a+1}
$$
thanks to $h(a)\le(a+1)/2$. This yields that $\ffi'(t)=h(a)-(a+1)h'(a)>0$, so $\ffi$ is
strictly increasing on $[0,1]$. Since $\ffi(t)=(2-t)\,\sigma\,t$ by symmetry of
$\sigma$, $\ffi(0)=2\,\sigma\,0=\gamma_0$. Also $\ffi(1)=1$. Hence one can define
$A:=\ffi^{-1}(Y)$ so that $A\in B(\cH)^{++}$, $A\le I$, and $Y=\ffi(A)=A\,\sigma\,(2I-A)$.
\end{proof}

When $\gamma_0=0$, for every $X,Y\in B(\cH)^{++}$ with $X\ge Y$ we have
$A,B\in B(\cH)^{++}$ such that $X=A\,\triangledown\,B$ and $Y=A\,\sigma\,B$. For example,
when $\sigma=\,!$ and $\#$, $A$ and $B$ can be chosen, respectively, as follows:
$$
\begin{cases}A=X-X\,\#\,(X-Y), \\ B=X+X\,\#\,(X-Y),\end{cases}\quad
\begin{cases}A=X-X\,\#\,(X-YX^{-1}Y), \\ B=X+X\,\#\,(X-YX^{-1}Y).\end{cases}
$$

\bigskip
\noindent{\it Proof of Proposition \ref{P-4.1}.}\enspace
The necessity of (1) and (2) for $f$ to be operator monotone is obvious. Assume (1) and
let $X,Y\in B(\cH)^{++}$ with $X\ge Y$. Choose a $\gamma\in(\gamma_0,1)$, where
$\gamma_0\in[0,1)$ be as in Lemma \ref{L-4.2}, and define for $k=0,1,2,\dots$
$$
X_k:=\gamma^kX+(1-\gamma^k)Y.
$$
Then $X_0=X$, and we have $X_k\ge X_{k+1}\ge\gamma X_k$ for each $k\ge0$ because
$$
X_k-X_{k+1}=(\gamma^k-\gamma^{k+1})(X-Y)\ge0,\quad
X_{k+1}-\gamma X_k=(1-\gamma)Y\ge0.
$$
Hence by Lemma \ref{L-4.2}, (1) implies that
$$
f(X)\ge f(X_1)\ge\cdots\ge f(X_k)\ge\cdots,\qquad k\ge1.
$$
Since $X_k-Y=\gamma^k(X-Y)\to0$ so that $f(X_k)\to f(Y)$ in the operator norm, we have
$f(X)\ge f(Y)$.

In the same way it follows that $f$ is operator monotone decreasing if and only if the
reversed inequality of (1) holds. Moreover, conditions (1) and (2) are transformed into
each other when $f$ is replaced by $f(x^{-1})$ and $\sigma$ by the adjoint $\sigma^*$.
Hence the assertions for (2) are immediate from those for (1).\qed

\section*{Acknowledgments}

The authors are grateful to Dr.\ Mil\'an Mosonyi for discussions. He proposed the question
about the convexity of the functional $\log\omega(A^\alpha)$, that is the starting point
of this work. They thank Professor Rajendra Bhatia for calling their attention to the
paper \cite{ARV} after the first version of the manuscript was completed. The work of
F.H.\ was partially supported by Grant-in-Aid for Scientific Research (C)21540208.

\end{document}